\RequirePackage[l2tabu, orthodox]{nag}

\documentclass[12 pt, oneside]{amsart}

\usepackage[margin=1.00in,foot=.5in]{geometry}
\usepackage{dutchcal}
\usepackage{amssymb,amsmath,amsthm}
\usepackage{amsmath}
\usepackage{mathtools}
\usepackage[justification=centering]{caption}
\usepackage{stmaryrd}
\usepackage[hidelinks]{hyperref}
\usepackage{mathrsfs}
\usepackage[shortlabels]{enumitem}

\newcommand{\Q}{\mathbb{Q}}

\newcommand{\N}{\mathbb{N}}

\newcommand{\set}[1]{\left\{#1\right\}}

\newcommand{\abs}[1]{\left| #1 \right| }
\newcommand{\OurLinearTerm}{\frac{99}{37}}
\newcommand{\OurConstantTerm}{\frac{187}{37}}

\theoremstyle{plain}

\newtheorem{theorem}[equation]{Theorem}

\newtheorem{lemma}[equation]{Lemma}
\newtheorem{corollary}[equation]{Corollary}
\newtheorem{proposition}[equation]{Proposition}

\theoremstyle{definition}
\newtheorem{definition}[equation]{Definition}
\newtheorem{exmp}[equation]{Example}

\usepackage[shortlabels]{enumitem}

\newenvironment{nouppercase}{%
  \renewcommand{\uppercasenonmath}[1]{}}{}

\numberwithin{equation}{section}
\numberwithin{figure}{section}

\title[Number of factors of an odd perfect number]{On inequalities involving counts of the prime factors of an odd perfect number}

\author{Graeme Clayton and Cody S. Hansen}

\begin{document}

\begin{nouppercase}
\maketitle
\end{nouppercase}

\begin{abstract}
Let $N$ be an odd perfect number.
Let $\omega(N)$ be the number of distinct prime factors of $N$ and let $\Omega(N)$ be the total number (counting multiplicity) of prime factors of $N$.
We prove that $\OurLinearTerm\omega(N) - \OurConstantTerm \leq \Omega(N)$ and that if $3\nmid N$, then $\frac{51}{19}\omega(N)-\frac{46}{19} \leq \Omega(N)$.
\end{abstract}

\section{Introduction}
A positive integer is said to be \textit{perfect} if the sum of its positive integer divisors is twice itself.
For the purposes of this paper, we will assume that $N$ is an odd perfect number, although it is a centuries-old open problem to prove that none exist.
It was shown by Euler that for any such $N$ we have
\[
N=p_0^{e_0}m^2,
\]
where $p_0\equiv e_0\equiv 1\pmod 4$ and $p_0$ is a prime not dividing $m$. We call $p_0$ the \textit{special} prime factor of $N$.

It is apparent from the previously displayed equation that the inequality
\begin{align}\label{Eq:TrivialIneq}
2\omega(N)-1 \leq \Omega(N)
\end{align}
holds, where $\omega(N)$ is the number of distinct prime factors of $N$ and $\Omega(N)$ is the total number (counting multiplicity) of prime factors of $N$.
Since we are only concerned with the prime-counting functions $\omega$ and $\Omega$ applied to the argument $N$, for simplicity we will hereafter suppress $N$ from the notation.

In a paper by Ochem and Rao \cite{O&R}, and in two papers by Zelinsky \cite{Zel1,Zel2}, inequalities of the form $a\omega+b\leq \Omega$, with $a,b \in \Q$ and $a>2$, are obtained as improvements to the trivial inequality \eqref{Eq:TrivialIneq}.
Prior to this paper, the best such (asymptotic) inequalities, as proven by Zelinsky, were
\[
\frac{66}{25}\omega-5 \leq \Omega, \qquad \text{if $3\mid N$},
\]
and
\[
\frac{302}{113}\omega-\frac{286}{113} \leq \Omega, \qquad \text{if $3\nmid N$.}
\]

These improvements over the trivial inequality \eqref{Eq:TrivialIneq} are based on the fact that $\sigma$, the sum of divisors function, is multiplicative.
Hence, letting the prime factorization of $N$ be given as  $N=p_0^{e_0}p_1^{e_1}\cdots p_k^{e_k}$, then
\begin{align*}
2p_0^{e_0}p_1^{e_1}\cdots p_k^{e_k}=2N = \sigma(N) = \sigma( p_0^{e_0}p_1^{e_1}\cdots p_k^{e_k})=\prod_{i=0}^k \sigma(p_i^{e_i}).
\end{align*}
From this we see that each odd prime $q$ dividing $N$ must divide some $\sigma(p_i^{e_i})$. Hence, we make the following definition.

\begin{definition}
Given an odd perfect number $N$, we say that a prime $p$ dividing $N$ \textit{contributes} a prime $q$ (or that $q$ is \textit{contributed} by $p$) if $q\mid\sigma(p^{e})$, where $p^e\mid\mid N$ (meaning that $p^e\mid N$ but $p^{e+1}\nmid N$).
\end{definition}

We observe that
\[
\sigma(p^{e}) =p^e+p^{e-1}+\cdots+p+1 = \frac{p^{e+1}-1}{p-1} = \prod_{\substack{d\mid (e+1)\\ d\neq 1}}\Phi_d(p),
\]
where $\Phi_d(x)$ is the $d^{\text{th}}$ cyclotomic polynomial.
Thus, the factorization of $N$ is closely linked to the factorization of cyclotomic polynomials.
Our main improvement to previous bounds is a result of discovering new restrictions on such factorizations.

We use these results to obtain a system of inequalities.
We then optimize that system to get the following theorem.
\begin{theorem}\label{Theorem:OurResult}
If $N$ is an odd perfect number, then
\[
\OurLinearTerm\omega - \OurConstantTerm \leq \Omega.
\]
\end{theorem}
This improves Zelinsky's bound in the case when $3\mid N$, but only asymptotically improves the bound when $3\nmid N$. However, we are able to easily modify the system of inequalities to handle this case, resulting in the next theorem, which is a strict improvement over Zelinsky's bound.

\begin{theorem}\label{Theorem:OurNot3Result}
If $N$ is an odd perfect number and $3\nmid N$, then
\[
\frac{51}{19}\omega - \frac{46}{19}\leq \Omega.
\]
\end{theorem}


\section{Definitions}

Following notation from the previously cited papers, as well as from the introduction, let $\omega$ be the number of distinct prime factors of $N$, and let $\Omega$ be the number of total prime factors of $N$.
Let $f_3$ be the integer where $3^{f_3}\mid\mid N$ (i.e., $f_3$ is the $3$-adic valuation of $N$).
Also, let $p_0$ be the special prime of $N$.

We will focus on prime divisors of $N$ that are neither $3$ nor the special prime. Therefore, we make the following notational choice:
\begin{definition}
Let
\[
P \coloneqq \set{p \text{ prime}\;:\;p\mid N,\ p\neq p_0,\ p\neq 3}.
\]
\end{definition}

We often need to speak of the set of the primes contributed by $P$.
Thus, we introduce the following notation.
\begin{definition}
Let
\[
Q \coloneqq \set{q\ \text{prime}\;:\; \text{$q$ is contributed by some $p\in P$}}
\]
\end{definition}

We define a function that describes what primes are contributed by a prime in $P$.
\begin{definition}
Let $f$ be the function from $P$ to the power set of $Q$, defined by the rule
\[
f(p) =\set{q\in Q \;:\; p \text{ contributes } q }.
\]
By an abuse of notation, for any subset $P'\subseteq P$ we also define
\[
f(P') = \bigcup_{p\in P'} f(p).
\]

\end{definition}
We now separate the primes in $P$, based on how many times they divide $N$.
\begin{definition}
Let
\[
S\coloneqq\set{p\in P\;:\; p^2 \mid \mid N }
\]
and
\[
T\coloneqq\set{p\in P\;:\; p^4\mid N}.
\]
Note that $S\cap T =\emptyset$ and $S\cup T = P$. Let $g_4$ be the number of primes factors $p$ of $N$, counting multiplicity as divisors of $N$, where $p\in T$.
\end{definition}

In order to further differentiate between elements of $S$, we introduce notation describing certain subsets of $S$.

\begin{definition}

Given an integer $m$, given sets $U_1,\dots, U_n$, with $n\leq m$, and given an integer $j\in\set{1,2}$, we let $S_{m,j}^{U_1,\dots, U_m}$ denote the set of all $p \in S$ that satisfy the following conditions:
\begin{enumerate}
    \item $p$ contributes exactly $m$ primes (counting multiplicity), call them $q_1,\dots, q_m$,
    \item up to reordering, $q_i\in U_i$ for each $1\leq i \leq n$, and
    \item $p\equiv j \pmod 3$.
\end{enumerate}
\end{definition}
We allow $n<m$ because we do not always need to focus on all of the contributed primes of elements of $S$.
Also, notice that we do not allow $j=0$, because $3 \not\in S$.

Note that whenever $j=1$, all elements of the set $S_{m,j}^{U_1,\dots,U_n}$ will contribute the prime $3$.
This fact is proved in Lemma \ref{Lemma:OnlyOne3}.
\begin{lemma}\label{Lemma:OnlyOne3}
If $p \in S_{m,1}^{U_1,\dots, U_n}$, then $p$ contributes $3$ exactly once.
\end{lemma}
\begin{proof}
Since $j=1$, we see that $p \equiv 1 \pmod 3$. As $p \in S$, we have
\[
\sigma(p^2)=p^2+p+1\equiv 1^2+1+1\equiv 0\pmod 3.
\]
To see that $p$ does not contribute 3 twice, note that $p$ is one of $1,4$, or $7$ modulo $9$ and in all cases $p^2+p+1$ is 3 modulo $9$. Thus, in no case can $p$ contribute 3 twice.
\end{proof}

We now illustrate our newly-defined notation with the following example.

\begin{exmp}

Suppose that $7^2,107^2,557^2\mid\mid N$ and that $13^4,127^4,6343^4 \mid N$.
We see that
\[
\sigma(557^2)= 557^2+557+1=7^2\cdot 6343,
\]
so 557 contributes 7 twice.
Since $557,7 \in S$ and $557 \equiv 2 \pmod 3$, this means that $557 \in S_{3,2}^{S,S}$. 
We also have $6343 \in T$, so $557 \in S_{3,2}^T$, and $557 \in S_{3,2}^{S,S,T}$.
Since
\[
\sigma(107^2)=107^2+107+1=7\cdot13\cdot127
\]
and $7 \in S$ we have that 107 is an element of all of the following sets:
\[
S_{3,2}^S,\ S_{3,2}^T,\ S_{3,2}^{T,T}, \text{ and } S_{3,2}^{S,T,T}.
\]
However, since  $13^4\mid N$ and $127^4\mid N$, we have $13,127 \in T$, so $13,127\not\in S$, meaning that $107 \not\in S_{3,2}^{S,S}$.
Similarly, since $7\in S$ we have $7 \not\in T$, so $557 \not \in S_{3,2}^{T,T}$.
\end{exmp}

In order to more easily discuss prime divisors of $N$ without regard to their value modulo $3$, we make the following definition.
\begin{definition}
Let
\[
S_m^{U_1,\dots,U_n} \coloneqq S_{m,1}^{U_1,\dots,U_n} \cup S_{m,2}^{U_1,\dots,U_n}.
\]
\end{definition}

Our methods work best when we consider small values of $m$.
Thus, we introduce the following notation to consolidate the contrary cases.
\begin{definition}
Let
\[
S_{\geq k,j}^{U_1,\dots, U_n}:=\bigcup_{m\geq k} S_{m,j}^{U_1,\dots, U_n},
\]
for a positive integer $k$.
\end{definition}

In \cite{Zel1}, Zelinsky proves that the following inequalities hold:
\begin{align*}
\abs{S_1}+\abs{S_{2,2}} &\leq \abs{T}+\abs{S_{2,1}}+\abs{S_{\geq 3,1}}+1 \\
\abs{S_1} &\leq \abs{T}+\abs{S_{\geq 3,1}}+1
\end{align*}
In doing so, he shows that if $p_1\in S_1$ and $p_2\in S_2$, then the largest contributed prime of $p_2$ is not contributed by $p_1$.
Furthermore, if $p_3,p_4\in S_{2,2}$, then the largest contributed prime of $p_3$ is also not the largest contributed prime of $p_4$.
These results lead to the following definition.
\begin{definition}
For each $p\in S_1 \cup S_2\cup S_{3,1}$, we define its \textit{linked prime} $\ell_p$ as follows.
Let $\ell_p$ be the largest prime contributed by $p$, except in the case where $p \in S_{2,2}$ and the largest prime contributed by $p$ is contributed by an element of $S_{2,1}$ as well.
In this special case, we will take $\ell_p$ to be the smaller prime contributed by $p$ instead.
\end{definition}

The inequalities above hold if the linking map $p\mapsto \ell_p$ is injective when considered separately over the domains $S_1\cup S_{2,1}$ and $S_1 \cup S_{2,2}$, respectively.
One of our main results is that the linking map is still injective over the union of these domains, which we prove in Lemma \ref{Lemma:LinkInjective}.

\section{Lemmas for linked primes}
We start with a well-known fact that was mentioned on page $2436$ of \cite{O&R}.
We leave the easy proof to the motivated reader.
\begin{lemma}\label{Lemma:ModularityOfSigma}
Let $a$, $b$, and $c$ be primes such that
\[
a\mid \sigma(b^{c-1}).
\]
Then, either $a=c$ or $a \equiv 1 \pmod c$.
In particular, if $c=3$, then either $a=3$ or $a\equiv 1\pmod 3$.
\end{lemma}

We will also often use the following simple fact without comment.
\begin{lemma}\label{Lemma:SimplifyingLemma}
Let $a$, $b$, $c$, and $d$ be positive integers with $b\geq c$.
If $a^2+a+1=bcd$, then
\[
b > \frac{a}{\sqrt{d}}.
\]

\end{lemma}

\begin{proof}
We have
\[
b^2 \geq bc = \frac{a^2+a+1}{d}
\]
and so
\[
b\geq\sqrt{\frac{a^2+a+1}{d}} > \frac{a}{\sqrt{d}}. \qedhere
\]
\end{proof}

The next lemma is a key tool that is used repeatedly in the proofs that follow.

\begin{lemma}\label{Lemma:Factorization1}
Let $a,b,c,d$, and $e$ be positive integers, with $c$ prime, that satisfy
\[
a^2+a+1=cd,
\]
\[
b^2+b+1=ce\text{, and}
\]
\[
c>d>e.
\]
Then, $c=a+b+1$  and $a-b=d-e$.
\end{lemma}
\begin{proof}
Since $c>d$ and $cd>a^2$, it follows that $c>a$. Note that
\[
c \mid (a^2+a+1)-(b^2+b+1)=(a-b)(a+b+1).
\]
As $c$ is prime, either $c\mid a-b$ or $c\mid a+b+1$.
Clearly, $c \nmid a-b$, as $c>a$.
Hence, $c\mid a+b+1$.
However, note that $2c>2a\geq a+b+1$, as $a>b$.
Therefore, we must have that $c=a+b+1$.
Then, since
\[
(a-b)c=(a-b)(a+b+1)=(a^2+a+1)-(b^2+b+1) = cd-ce = c(d-e),
\]
we have that $a-b=d-e$. \qedhere

\end{proof}

We next show that the largest prime contributed by an element of $S_2$ is not contributed by an element of $S_1$.

\begin{lemma}[{cf.\ \cite[Lemma 3]{O&R}}]\label{Lemma:ZelProof1}
Let $a$, $b$, $c$, and $d$ be positive integers, with $c>d>1$, and with $c$ prime. If
\[
    a^2+a+1=cd \quad \text{and}
\]
\[
    b^2+b+1=c,
\]
then $a=b^2$.
In particular, $a$ is not prime in this case.
\end{lemma}

\begin{proof}
Taking $e=1$, we can apply Lemma \ref{Lemma:Factorization1} to get that $c=a+b+1$.
This implies that
\[
b^2+b+1=a+b+1,
\]
so $a=b^2$.
Thus $a$ is not prime.
\end{proof}

We next show that two distinct elements of $S_{2,2}$ never share their largest contributed prime.
\begin{lemma}[{cf.\ \cite[Lemma 3]{O&R}}]\label{Lemma:ZelProof2}
Let $a$, $b$, $c$, $d$, and $f$ be odd primes greater than $3$, with $c>d>f$.
Then, it does not hold that
\[
a^2+a+1=cd \ \text{and}
\]
\[
b^2+b+1=cf.
\]
\end{lemma}

\begin{proof}
By way of contradiction, suppose that the equalities do hold.
Then, since $3\nmid a^2+a+1$, we have that $a\equiv 2\pmod 3$.
Similarly, $b\equiv 2 \pmod 3$.
By applying Lemma \ref{Lemma:Factorization1} with $e=f$, we have $a+b+1=c$, and so
\[
1 \equiv c = a+b+1 \equiv 2\pmod 3,
\]
a contradiction.
\end{proof}

Two distinct elements of $S_{2,1}$ clearly have distinct largest contributed primes.
However, the case remains when a prime $p_1\in S_{2,1}$ and another prime in $p_2 \in S_{2,2}$ may have the same largest contributed prime.
For example, if $p_1=7$ and $p_2=11$, then we have $\sigma(7^2)=3\cdot19$ and $\sigma(11^2) = 7\cdot19$.
Nevertheless, we prove below that if this occurs, the smaller contributed prime from $p_2$ is distinct from all of the largest contributed primes from $S_1\cup S_2$.
Further, we will see in Corollary \ref{Corollary:Uniqueness} that the smaller contributed prime from $p_2$ is not contributed in this same way by another similar pair of primes.

\begin{lemma}\label{Lemma:Unique between S1 and S2}
Suppose we have odd primes $a$, $b$, $c$, and $d$, with $c>d$ and $d\neq 3$, that satisfy the equations
\[
a^2+a+1=cd \quad \text{and}
\]
\[
b^2+b+1=3c.
\]
Then, there does not exist any odd prime $g$ such that $g^2+g+1=dh$, with $d>h$, where $h=1$ or $h$ is an odd prime.
\end{lemma}

\begin{proof}
By Lemma \ref{Lemma:Factorization1}, we have $c=a+b+1$ and $d-3=a-b$.
Then, we have
\begin{align*}
    b^2+b+1&=3c =3(a+b+1)=3a+3b+3,
\end{align*}
and so $a=(b^2-2b-2)/ 3$.
We also have
\begin{align}\label{Eq:QuadraticEquation}
    d&= \frac{b^2-2b-2}{3}-b+3= \frac{b^2-5b+7}{3}.
\end{align}
Then, suppose $g$ is an odd prime such that $g^2+g+1=dh$ with $d>h$.
Then, we have three cases.

\vspace{5mm}

\noindent\textbf{Case 1}: Suppose that $g \equiv 0 \pmod 3$.
Then, $g=3$, so we find that $d=13$ and $h=1$.
We know that $d=(b^2-5b+7)/3$, so $0=b^2-5b-32$.
But this has no integer solutions, a contradiction.

\vspace{5mm}

\noindent\textbf{Case 2}: Suppose that $g\equiv 1\pmod 3$.
Then, $g^2+g+1\equiv 0\pmod 3$, so
\[
g^2+g+1=3d=b^2-5b+7.
\]
Solving for $g$ with the quadratic equation, we find that the only positive solution is $g=b-3$.
However, this contradicts the fact that $g$ is odd.

\vspace{5mm}

\noindent\textbf{Case 3}: Suppose that $g\equiv 2\pmod 3$.
We will show that $2d>a$. To that end, observe that
\[
2d-a = \frac{2b^2-10b+14}{3}-\frac{b^2-2b-2}{3} = \frac{b^2-8b+16}{3}=\frac{(b-4)^2}{3},
\]
which is greater than $0$, since $b\neq 4$.
Since
\[
g^2+g+1 =dh <d^2 <cd = a^2+a+1
\]
we have $g<a$.
Thus $2d>a$ and we have that
\begin{equation}\label{Eq:a+g+1Bound}
a+g+1 < 2a < 4d
\end{equation}
Hence, from our earlier factorization, we have
\[
d(c-h)=cd-dh = (a^2+a+1)-(g^2+g+1) = (a-g)(a+g+1),
\]
which is positive.
Hence, $d\mid a-g$ or $d\mid a+g+1$.
If $d\mid a-g$, then $d=a-g$, as $2d>a$.
However, this means $d \equiv 0 \pmod 3$, a contradiction.
Thus, we have $d\mid a+g+1$, i.e., $kd=a+g+1$ for some $k\in \N$.
We see that $a\equiv 5\pmod 6$ and $g \equiv 5\pmod 6$.
Thus, $a+g+1\equiv 5\pmod 6$.
Thus, $k \equiv 5 \pmod 6$ since $d\equiv 1\pmod 6$.
However, this means that $5d \leq a+g+1$, contradicting \eqref{Eq:a+g+1Bound}.
\end{proof}

The following corollary shows that in the exceptional case of the definition of linked primes, the smaller contributed prime uniquely determines the other primes.

\begin{corollary}\label{Corollary:Uniqueness}
Given an odd prime $d$, if there exist odd primes $a,b,$ and $c$ that satisfy
\begin{itemize}
     \item[\textup{(1)}] $a^2+a+1=cd$,
     \item[\textup{(2)}] $b^2+b+1=3c$, and
     \item[\textup{(3)}] $c>d$,
\end{itemize}
they are unique.

\end{corollary}

\begin{proof}

From \eqref{Eq:QuadraticEquation} we get that $d=(b^2-5b+7)/3$.
By the quadratic equation,
\[
b=\frac{1}{2} \left(5\pm \sqrt{12 d-3}\right).
\]
Since $b>0$ and $d>3$, the only solution is
\[
b=\frac{1}{2}\left(5+\sqrt{12d-3} \right).
\]
Then, since $b^2+b+1=3c$, we can write $c$ entirely in terms of $d$.
Likewise, since $a^2+a+1=cd$, we can solve for $a$ in terms of $d$ using the quadratic equation.
Since $a>0$, we have
\[
a=\frac{1}{2} \left(-1 + \sqrt{4 d^2+4d \sqrt{12 d-3} +12 d-3}\right).
\]
Thus, $a$, $b$, and $c$ are all uniquely determined by $d$.
\end{proof}

Using the above lemmas we show that the linking map from $S_1\cup S_2$ to $Q$ is injective.
\begin{lemma}\label{Lemma:LinkInjective}
The linking map $ \ell: S_1 \cup S_2 \to Q$ defined by the rule $p \mapsto \ell_p$ is injective.
\end{lemma}

\begin{proof}
Suppose that we have $p_1,p_2 \in S_1 \cup S_2$ and suppose that $\ell_{p_1}=\ell_{p_2}$.
Our goal is to show that $p_1=p_2$.
We have nine cases to consider.

\vspace{5mm}

\noindent\textbf{Case 1}: Suppose that $p_1,p_2 \in S_1$.
We have that
\[
p_1^2+p_1+1 = \ell_{p_1} = \ell_{p_2} = p_2^2+p_2+1,
\]
and thus $p_1=p_2$ since $p_1,p_2>0$.

\vspace{5mm}

\noindent\textbf{Case 2}: Suppose that $p_1,p_2 \in S_{2,1}$.
We have that
\[
p_1^2+p_1+1 = 3\ell_{p_1} = 3\ell_{p_2} = p_2^2+p_2+1,
\]
and thus $p_1=p_2$ since $p_1,p_2>0$.

\vspace{5mm}

\noindent\textbf{Case 3}: Suppose that $p_1,p_2 \in S_{2,2}$ and suppose that $\ell_{p_1}$ is the largest prime divisor of $\sigma(p_1^2)$ and that $\ell_{p_2}$ is the largest divisor of $\sigma(p_2^2)$.
This case is handled by Lemma \ref{Lemma:ZelProof2}.

\vspace{5mm}

\noindent\textbf{Case 4}: Suppose that $p_1,p_2 \in S_{2,2}$ and suppose that $\ell_{p_1}$ is the largest prime divisor of $\sigma(p_1^2)$ and that $\ell_{p_2}$ is the smallest divisor of $\sigma(p_2^2)$.
By the definition of linked primes, this means that there is some $p_3 \in S_{2,1}$ that shares its largest prime divisor with $\sigma(p_2^2)$.
Hence, this case is ruled out by Lemma \ref{Lemma:Unique between S1 and S2}.

\vspace{5mm}

\noindent\textbf{Case 5}: Suppose that $p_1,p_2 \in S_{2,2}$, and suppose that $\ell_{p_1}$ and $\ell_{p_2}$ are the smallest prime divisors of $\sigma(p_1^2)$ and $\sigma(p_2^2)$ respectively.
By the definition of linked primes, this means that there is some $p_3 \in S_{2,1}$ that shares its largest prime divisor with $\sigma(p_1^2)$.
Similarly, there is some $p_4 \in S_{2,1}$ that shares its largest prime divisor with $\sigma(p_2^2)$.
Hence, by Corollary \ref{Corollary:Uniqueness} we must have $p_1=p_2$.

\vspace{5mm}

\noindent\textbf{Case 6}: Suppose that $p_1\in S_1$ and $p_2 \in S_{2,1}$.
This case is ruled out by Lemma \ref{Lemma:ZelProof1}, by taking $d=3$.

\vspace{5mm}

\noindent\textbf{Case 7}: Suppose that $p_1 \in S_1$ and $p_2 \in S_{2,2}$, and suppose that $\ell_{p_2}$ is the largest prime divisor of $\sigma(p_2^2)$.
This case is also ruled out by Lemma \ref{Lemma:ZelProof1}.

\vspace{5mm}

\noindent\textbf{Case 8}: Suppose that $p_1 \in S_1$ and $p_2 \in S_{2,2}$, and suppose that $\ell_{p_2}$ is the smallest prime divisor of $\sigma(p_2^2)$.
Then, there exists some $p_3 \in S_{2,1}$  such that $\sigma(p_3^2)$ shares its largest prime divisor with $\sigma(p_2^2)$.
Hence, this case is ruled out by Lemma \ref{Lemma:Unique between S1 and S2}, taking $h=1$.

\vspace{5mm}

\noindent\textbf{Case 9}: Suppose that $p_1\in S_{2,1}$ and $p_2 \in S_{2,2}$.
From the definition of linked primes, this means that $\ell_{p_2}$ is the smallest prime divisor of $\sigma(p_2^2)$.
Hence, there exists some $p_3 \in S_{2,1}$ such that $\sigma(p_3^2)$ shares its largest prime divisor with $\sigma(p_2^2)$.
We see that this is impossible due to Lemma \ref{Lemma:Unique between S1 and S2}, with $h=3$ and $d=\ell_{p_1}$.
\end{proof}

\section{A generalization of the linking map}

As linked primes are one of the most important tools we use in establishing our main inequality between $\Omega$ and $\omega$, we naturally want to extend the domain of the linking map.
However, we were not able to establish an injective map that links all primes in $S$ to their contributed primes in $Q$.

Furthermore, in \cite{H&N} it was shown that a linking map $S_1\cup S_{4,2} \to Q$ may not be injective.
Thus it is hopeless to expect to generalize our injective linking map over the entire domain $S$.
However, we are able to show that our linking map is nearly injective over the domain $S_1\cup S_2\cup S_{3,1}$, in the sense that for every contributed prime $q$, we have that $\abs{\set{p\in S_1\cup S_2 \cup S_{3,1}\; : \; \ell_p=q}} \leq 2$.

We now prove two lemmas analogous to Lemma \ref{Lemma:Factorization1} for use with primes from $S_{3,1}$.
\begin{lemma}\label{Lemma:Factorization2}
Let $a,b,d,e$, and $f$ be positive integers, with $a,b,d$ prime, that satisfy
\begin{align*}
    a^2+a+1=3de \quad \text{and}\\
    b^2+b+1=df
\end{align*}
with $d>e>f$.
Then, $d=a+b+1$ and $a-b=3e-f$.
\end{lemma}

\begin{proof}
From Lemma \ref{Lemma:SimplifyingLemma}, we have that $d>a/\sqrt{3}>a/2$.
Then,
\[
d(3e-f) =3de-df=(a^2+a+1)-(b^2+b+1) = (a-b)(a+b+1).
\]
Hence, either $d\mid (a-b)$ or $d\mid (a+b+1)$.

First, suppose $d\mid (a-b)$.
Then, since $2d>a$, we have that $d=a-b$.
But $1\equiv d =a-b \equiv 2\pmod 3$, a contradiction, so this case cannot occur.

Now, suppose $d\mid (a+b+1)$.
Then, $a+b+1\leq 2a < 4d$ and so $kd=a+b+1$ for some $k\in\set{1,2,3}$.
But $k=3$ yields $3d=a+b+1\equiv 1\pmod 3$, a contradiction.
Similarly, $k=2$ gives $2d=a+b+1\equiv 1\pmod 2$, a contradiction.
Thus $d=a+b+1$, and therefore $a-b=3e-f$.
\end{proof}

\begin{lemma}\label{Lemma:Factorization3}
Let $a,b,d,e,$ and $f$, be positive integers, with $a,b,d$ prime, that satisfy
\begin{align*}
    a^2+a+1=3de \quad \text{and}\\
    b^2+b+1=3df
\end{align*}
with $d\geq e>f$.
Then, $d=(a+b+1)/3$ and $a-b=e-f$.
\end{lemma}

\begin{proof}
From Lemma \ref{Lemma:SimplifyingLemma}, $d\geq a/\sqrt{3}>a/3$.
Then, we have
\[
3d(e-f) =3de-3df=(a^2+a+1)-(b^2+b+1) = (a-b)(a+b+1).
\]
Note that in this case we have $a\equiv b \equiv 1 \pmod 3$.
Hence, $a-b\equiv a+b+1\equiv 0 \pmod 3$, so $d\mid (a-b)/3$ or $d\mid (a+b+1)/3$.
However, $d>a/3>(a-b)/3$, so $d\nmid (a-b)/3$.
It follows that $d\mid (a+b+1)/3$.
Note that
\[
\frac{a+b+1}{3} \leq \frac{2a}{3} < 2d.
\]
Thus, $d=(a+b+1)/3$, and therefore $a-b=e-f$.
\end{proof}

We now show that the largest contributed prime of an element of $S_{3,1}$ is contributed by at most one other element of $S_{3,1}$

\begin{lemma}\label{Lemma:Semi-linked in S31}
There do not exist distinct odd primes $a$, $b$, $c$, $d$, $e$, $f$, and $g$ that satisfy
\begin{enumerate}
    \item[\textup{(1)}] $a^2+a+1=3de$,
    \item[\textup{(2)}] $b^2+b+1=3df$,
    \item[\textup{(3)}] $c^2+c+1=3dg$, and
    \item[\textup{(4)}] $d\geq e>f>g$.
\end{enumerate}
\end{lemma}

\begin{proof}
Suppose to the contrary that such primes exist.
By Lemma \ref{Lemma:Factorization3} we have that $d=(a+b+1)/3$ and $d=(a+c+1)/3$.
Thus $b=c$ and so $f=g$, a contradiction.
\end{proof}

We next show that the largest contributed prime of an element from $S_{2,2}$ can only be the largest contributed prime of a single element from $S_{3,1}$.

\begin{lemma}\label{Lemma:Semi-linked between S22 and S31}
There do not exist odd primes $a$, $b$, $c$, $d$, $e$, $f$, and $g$ that satisfy
\begin{enumerate}
    \item[\textup{(1)}] $a^2+a+1=3de$,
    \item[\textup{(2)}] $b^2+b+1=3df$,
    \item[\textup{(3)}] $c^2+c+1=dg$,
    \item[\textup{(4)}] $d\geq e,f,g$, and
    \item[\textup{(5)}] $e\neq f$.
\end{enumerate}
\end{lemma}

\begin{proof}
Suppose to the contrary that such primes exist.
By Lemma \ref{Lemma:Factorization2} we have that $d=a+c+1$ and $d=b+c+1$.
Thus $a=b$ and so $e=f$, a contradiction.
\end{proof}

We now show that the largest contributed prime of an element from $S_{3,1}$ is not also the contributed prime of an element from $S_1$.

\begin{lemma}\label{Lemma:Unique between S1 and S31}
There do not exist odd primes $a$, $b$, $d$, and $e$ that satisfy
\begin{enumerate}
    \item[\textup{(1)}] $a^2+a+1=3de$,
    \item[\textup{(2)}] $b^2+b+1=d$, and
    \item[\textup{(3)}] $d\geq e$.
\end{enumerate}

\end{lemma}

\begin{proof}
Suppose to the contrary that such primes exist.
By Lemma \ref{Lemma:Factorization2}, with $f=1$, we have that $d=a+b+1$.
Thus,
\[
b^2+b+1= a+b+1
\]
which implies $a=b^2$, a contradiction.
\end{proof}

We next show that shows that the largest contributed prime of an element from $S_{3,1}$ is not also the largest contributed prime of an element from $S_{2,1}$.

\begin{lemma}\label{Lemma:Unique between S21 and S31}
There do not exist odd primes $a$, $b$, $d$, and $e$ that satisfy
\begin{enumerate}
    \item[\textup{(1)}] $a^2+a+1=3de$,
    \item[\textup{(2)}] $b^2+b+1=3d$, and
    \item[\textup{(3)}] $d\geq e$.
\end{enumerate}

\end{lemma}

\begin{proof}
Suppose to the contrary that such primes exist.
By Lemma \ref{Lemma:Factorization3}, with $f=1$, we have that $d=a+b+1$ and $d=(a-b+1)/3$.
Then,
\begin{align*}
    a^2+a+1 &= 3de \\
    &=3(a+b+1)\left(\frac{a-b+1}{3} \right) \\
    &=(a+b+1)(a-b+1) \\
    &=a^2+a+1+(a-b^2).
\end{align*}
Thus, $a=b^2$, a contradiction.
\end{proof}

We next show that in the exceptional case that an element of $S_{2,2}$ shares its largest contributed prime with an element of $S_{2,1}$, the smaller contributed prime of that element of $S_{2,2}$ is not the largest contributed prime of an element from $S_{3,1}$

\begin{lemma}\label{Lemma:Small factor of S21-S22 pair not S31}
Let $a$, $b$, $d$, and $f$ be odd primes such that $d>f>3$, satisfying
\[
    a^2+a+1=df \quad \text{and}
\]
\[
    b^2+b+1=3d.
\]
Then, there do not exist odd primes $c$ and $g$ such that
\[
c^2+c+1=3fg,
\]
with $f>g$.
\end{lemma}

\begin{proof}
Suppose the contrary that such primes exist.
Then, by Lemma \ref{Lemma:Factorization1}, we have that $f=a-b+3$.
From the proof of Corollary \ref{Corollary:Uniqueness}, we can express $a$ in terms of $f$.
We see that since $f>3$ we have
\begin{align*}
    a &= \frac{1}{2} \left(-1 + \sqrt{4 f^2+4f \sqrt{12 f-3} +12 f-3}\right) \\
    & <2f.
\end{align*}
From Lemma \ref{Lemma:SimplifyingLemma}, we have $f\geq c/\sqrt{3}$.

We can now factor whichever is positive of $(a^2+a+1)-(c^2+c+1)$ and $(c^2+c+1)-(a^2+a+1)$, giving us two cases to consider.

\vspace{5mm}

\noindent \textbf{Case 1}: Suppose $a>c$.
Thus we have
\[
f(d-3g) = df-3fg =  (a^2+a+1)-(c^2+c+1) = (a-c)(a+c+1).
\]
So $f\mid (a-c)$ or $f\mid (a+c+1)$.
Suppose $f\mid (a-c)$.
Since we also have $f>a/2$ from above, we have that $f=a-c$.
Then, $a-c=a-b+3$ and so $b=c+3$, a contradiction.
Thus $f\mid (a+c+1)$ and since
\[
a+c+1 \leq 2a <4f,
\]
we have that $kf=a+c+1$ for some $k\in \set{1,2,3}$.
Then, $k=3$ yields $3f=a+c+1\equiv 1\pmod 3$, a contradiction.
Similarly, $k=2$ gives $2f=a+c+1\equiv 1\pmod 2$, a contradiction.
Therefore, $k=1$ and so $f=a+c+1$.
Hence, $a+c+1=a-b+3$ and so $b+c=2$, a contradiction.

\vspace{5mm}

\noindent \textbf{Case 2}:
Suppose $a<c$.
Thus we have
\[
f(3g-d) = 3fg-df =  (c^2+c+1)-(a^2+a+1) = (c-a)(a+c+1).
\]
So $f\mid (c-a)$ or $f\mid (a+c+1)$.
Suppose $f\mid (c-a)$.
From Lemma \ref{Lemma:SimplifyingLemma} we have that $f>c/\sqrt{3}$, so $f>c/2$ and thus we have that $1\equiv f=c-a \equiv 2\pmod 3$, a contradiction.
Thus $f\mid (a+c+1)$ and since
\[
a+c+1 \leq 2c <4f,
\]
we have that $kf=a+c+1$ for some $k\in \set{1,2,3}$.
Then, $k=3$ yields $3f=a+c+1\equiv 1\pmod 3$, a contradiction.
Similarly, $k=2$ gives $2f=a+c+1\equiv 1\pmod 2$, a contradiction.
Thus $k=1$ and so $f=a+c+1$.
Hence, $a+c+1=a-b+3$ and so $b+c=2$, a contradiction.
\end{proof}

\begin{lemma}\label{Lemma:LinkAlmostInjective}
Let the linking map $\ell : S_1\cup S_2\cup S_{3,1} \to Q$ be defined by the rule $p \mapsto \ell_p$.
Then, for every $q\in Q$, $\abs{\set{p\in S_1\cup S_{2,1}\; : \; \ell_p=q}} =1$ and $\abs{\set{p\in S_{2,2} \cup S_{3,1}\; : \; \ell_p=q}} \leq 2$.
\end{lemma}

\begin{proof}

Suppose that we have $p_1,p_2,p_3 \in S_1 \cup S_2\cup S_{3,1}$, and suppose that $\ell_{p_1}=\ell_{p_2}=\ell_{p_3}$.
All nine cases in Lemma \ref{Lemma:LinkInjective} still apply, so we have five additional cases to consider.

\vspace{5mm}

\noindent\textbf{Case 10}: Suppose that $p_1,p_2,p_3\in S_{3,1}$.
Then, by Lemma \ref{Lemma:Semi-linked in S31}, we have, without loss of generality, that $p_1=p_3$.
Thus the preimage of $\ell_{p_1}$ has at most two elements, $p_1$ and $p_2$.

\vspace{5mm}

\noindent\textbf{Case 11}: Suppose that $p_1 \in S_1$ and $p_2 \in S_{3,1}$.
This case is ruled out by Lemma \ref{Lemma:Unique between S1 and S31}.

\vspace{5mm}

\noindent\textbf{Case 12}: Suppose that $p_1 \in S_{2,1}$ and $p_2 \in S_{3,1}$.
This case is ruled out by Lemma \ref{Lemma:Unique between S21 and S31}.

\vspace{5mm}

\noindent\textbf{Case 13}: Suppose that $p_1\in S_{2,2}$, $p_2\in S_{3,1}$, and $\ell_{p_1}$ is the largest contributed prime of $p_1$.
By the injectivity of the linking map over $S_1 \cup S_2$, we know that $p_3 \not \in S_1 \cup S_2$.
Hence, the only remaining possibility is that $p_3 \in S_{3,1}$, however this case is ruled out by Lemma \ref{Lemma:Semi-linked between S22 and S31}.

\vspace{5mm}

\noindent\textbf{Case 14}: Suppose that $p_1\in S_{2,2}$, $p_2\in S_{3,1}$, and $\ell_{p_2}$ is the smaller contributed prime of $p_1$.
This case is ruled out by Lemma \ref{Lemma:Small factor of S21-S22 pair not S31}.
\end{proof}

\section{Proof of Theorem \ref{Theorem:OurResult} and Theorem \ref{Theorem:OurNot3Result}}

To obtain the bound in \cite{Zel2}, twenty-one inequalities were derived.
We use thirteen of those inequalities, modify one to eliminate unnecessary variables, and improve one other (Equation 35 in \cite{Zel2}).
Further, we introduce four new inequalities.
The correspondence between this paper and \cite{Zel2} is shown in Table \ref{Table:EquationCorr} below.
We now give brief explanations of the inequalities.

\begin{table}[htbp]

    \begin{tabular}{|c|c|}

    \hline

       Our Equation & Corr.\ eqn.\ in \cite{Zel2}  \\

    \hline
        \eqref{Eq:SpecialExponent} & (26)  \\

        \eqref{Eq:SSplit} & (27)   \\

        \eqref{Eq:S2Split} & (28)   \\

        \eqref{Eq:S3Split} & (29) \\

        \eqref{Eq:SHSplit} & (30) \\

        \eqref{Eq:TotalFactors} & (31) \\

        \eqref{Eq:S31SplitMore} & (32)  \\

        \eqref{Eq:S31STSplitMore} & (33)  \\

        \eqref{Eq:S1SplitMore} & (34)   \\

        \eqref{Eq:BetterZel} & (35) \\

        \eqref{Eq:Contributed3's} & (36)  \\

        \eqref{Eq:ContributedByS} & (37) \\

        \eqref{Eq:Size of T} & (38)   \\



        \eqref{Eq:SpecialContributions} & (41)  \\




        \eqref{Eq:DistinctFactors} & (47) \\
        \hline
    \end{tabular}
    \caption{}
    \label{Table:EquationCorr}

\end{table} 


Inequalities \eqref{Eq:TotalFactors} through \eqref{Eq:SpecialExponent} are all straightforward consequences of their associated definitions.

\begin{align}
    \label{Eq:TotalFactors}e_0+f_3+2\abs{S}+g_4&=\Omega \\
    \label{Eq:DistinctFactors} \omega &\leq 2+\abs{S}+\abs{T} \\
    \label{Eq:Size of T}4\abs{T} &\leq g_4 \\
    \label{Eq:SpecialExponent}1&\leq e_0
\end{align}

\vspace{20mm}

Equations \eqref{Eq:SSplit} through \textbf{\eqref{Eq:S31SplitMore}} are obtained by decomposing $S$, and some of its subsets, into disjoint subsets.
Note there are similar decompositions for $S_{2,1}$ and $S_{2,2}$, but they do not contribute to our result.

\begin{align}
    \label{Eq:SSplit} \abs{S} &= \abs{S_1}+\abs{S_2}+\abs{S_3}+\abs{S_{\geq 4 }} \\
    \label{Eq:S2Split} \abs{S_2}&= \abs{S_{2,1}}+\abs{S_{2,2}} \\
    \label{Eq:S3Split} \abs{S_3}&= \abs{S_{3,1}}+\abs{S_{3,2}} \\
    \label{Eq:SHSplit} \abs{S_{\geq 4 }}&= \abs{S_{\geq 4 ,1}}+\abs{S_{\geq 4 ,2}} \\
    \label{Eq:S1SplitMore}\abs{S_1} &= \abs{S_1^S}+\abs{S_1^T}+\abs{S_1^{\set{p_0}}} \\
    \label{Eq:S31SplitMore}\abs{S_{3,1}}  &= \abs{S_{3,1}^{S,S}}+\abs{S_{3,1}^{T\cup\set{p_0},T\cup\set{p_0}}}+\abs{S_{3,1}^{S,T\cup \set{p_0}}} 
\end{align}

Inequalities \eqref{Eq:S31STSplitMore} and \eqref{Eq:S32SplitMore} are obtained by splitting sets into subsets, which may overlap.
In particular, it was proven in \cite{H&N} that for $p\in S_{3,2}$, with $f(p)=\set{p_1,p_2,p_3}$, that, without loss of generality, $p_1\notin f(S_1)$.
We make a distinction for when $p_1\in S$ and when $p_1\in T\cup \set{p_0}$, giving \eqref{Eq:S32SplitMore}.

\begin{align}
    \label{Eq:S31STSplitMore}\abs{S_{3,1}}  &\leq  \abs{S_{3,1}^{S\setminus f(S_1),T\cup \set{p_0}}}+\abs{S_{3,1}^{S,(T\cup \set{p_0})\setminus f(S_1)}} \\
    \label{Eq:S32SplitMore} \abs{S_{3,2}} &\leq
    \abs{S_{3,2}^{S\setminus f(S_1)}}+
    \abs{S_{3,2}^{(T\cup \set{p_0})\setminus f(S_1)}}
\end{align}


We obtain inequality \eqref{Eq:SpecialContributions} by noting that only one element of $S_1$ can contribute the special prime, i.e., $S_1^{\set{p_0}}$ has at most one element.

\begin{align}
    \label{Eq:SpecialContributions}\abs{S_1^{\set{p_0}}} &\leq 1 
\end{align}

Inequality \eqref{Eq:Contributed3's} holds since every element of $S_{2,1}\cup S_{3,1}\cup S_{\geq 4,1}$ contributes exactly one $3$, and thus $3$ must divide $N$ at least once for each element of these sets.

\begin{align}\label{Eq:Contributed3's}
\abs{S_{2,1}}+\abs{S_{3,1}}+\abs{S_{\geq 4 ,1}} \leq f_3
\end{align}

Counting the number of primes that are not $3$ contributed by elements of $S$ gives
\[
\abs{S_1}+2\abs{S_{2,2}}+3\abs{S_{3,2}}+4\abs{S_{\geq 4 ,2}}+\abs{S_{2,1}}+2\abs{S_{3,1}}+3\abs{S_{\geq 4 ,1}}  \leq g_4+e_0+2\abs{S_{2,1}}+2\abs{S_{3,1}}+2\abs{S_{\geq 4 ,1}},
\]
which simplifies to
\begin{align}\label{Eq:ContributedByS}
\abs{S_1}+2\abs{S_{2,2}}+3\abs{S_{3,2}}+4\abs{S_{\geq 4 ,2}}+\abs{S_{\geq 4 ,1}} \leq g_4+e_0+\abs{S_{2,1}}.
\end{align}


Inequality \eqref{Eq:BetterZel} comes from Lemma \ref{Lemma:LinkInjective} and inequality \eqref{Eq:NewIneq} comes from Lemma \ref{Lemma:LinkAlmostInjective}.

\begin{align}
\label{Eq:BetterZel}
    \abs{S_1}+\abs{S_2} &\leq \abs{T}+\abs{S_{2,1}}+\abs{S_{3,1}}+\abs{S_{\geq 4 ,1}}+1 \\
\label{Eq:NewIneq}
    \abs{S_1}+\abs{S_{2,1}}+\frac{1}{2}(\abs{S_{2,2}}+\abs{S_{3,1}}) &\leq \abs{T}+\abs{S_{2,1}}+\abs{S_{3,1}}+\abs{S_{\geq 4 ,1}}+1
\end{align}

Inequality \eqref{Eq:ContributedToSRedux} comes from establishing a lower bound for the total number of primes (counting multiplicity) which divide $N$ exactly twice and are $1$ modulo $3$.
The right hand side of the inequality is exactly the total number of such primes (note that no elements of $S_1$ are $1$ modulo $3$).

To justify the left hand side, note that every element of $S_1^S$ will contribute one unique prime to $S$ which is $1$ modulo $3$.
Since these primes are in $S$, by definition they must divide $N$ exactly twice.
We will show that all other sets used in the left-hand side of this inequality contribute at least one prime that is distinct from all of these primes contributed by $S_1^S$.
Therefore, we are allowed to double the $\abs{S_1^S}$ term in the inequality.

By Lemma \ref{Lemma:Unique between S1 and S31}, the largest prime contributed by an element of $S_{3,1}$ is not also contributed by an element of $S_1$.
Therefore, each prime in $S_{3,1}^{S,S}$ will contribute at least one prime which is in $S$ and is $1$ modulo $3$ and is distinct from any primes contributed by elements of $S_1$.

\begin{align}
    \label{Eq:ContributedToSRedux} 2\abs{S_1^S}+\abs{S_{3,1}^{S,S}}+
    \abs{S_{3,1}^{S\setminus f(S_1),T\cup \set{p_0}}}+\abs{S_{3,2}^{S\setminus f(S_1)}} &\leq 2\abs{S_{2,1}}+2\abs{S_{3,1}}+2\abs{S_{\geq 4 ,1}} 
\end{align}

Similarly, we have at least one such prime contributed by each element of $S_{3,1}^{S\setminus f(S_1),T\cup \set{p_0}}$ and by each element of $S_{3,2}^{S\setminus f(S_1)}$, simply by the definition of these sets.

Inequality \eqref{Eq:ContributedToTRedux} is very similar to \eqref{Eq:ContributedToSRedux}. It comes from establishing a lower bound for the total number of primes (counting multiplicity) which divide $N$ at least four times, or are equal to the special prime.
The argument justifying \eqref{Eq:ContributedToTRedux} is identical to the above in almost every way, with $T\cup\set{p_0}$ and $S$ interchanged, though we do use $S_1^T$ rather than $S_1^{T\cup\set{p_0}}$, simply to ensure that each element of the set contributes a prime which divides $N$ at least four times.

There is a notable difference between this inequality and the last one, however, in that $g_4+e_0$ is the total number of all primes which divide $N$ at least four times, or are the special prime, not just those that are $1$ modulo $3$.
We have not utilized the same subdivisions of the set $T$ as we have for $S$, but this does not affect the argument.

\begin{align}
    \label{Eq:ContributedToTRedux} 4\abs{S_1^T}+\abs{S_{3,1}^{T\cup\set{p_0},T\cup\set{p_0}}}+\abs{S_{3,1}^{S,(T\cup \set{p_0})\setminus f(S_1)}}+\abs{S_{3,2}^{(T\cup \set{p_0})\setminus f(S_1)}} &\leq g_4+e_0
\end{align}


Our goal is to get an inequality of the form $a\omega+b\leq \Omega$, with $a,b\in \Q$, with $a$ as large as possible.
Once we have found the maximum value of $a$, we would like to then maximize $b$.
To that end, we rewrite all inequalities (and equalities) of this section so that the right hand side is zero. 
We represent the left hand side of each rewritten inequality (5.i) by $x_i$, for $1\leq i \leq 19$.
We multiply each $x_i$ by a coefficient $c_i$ to get a linear combination
\begin{equation} \label{Eq:LinearCombination}
\sum_{i=1}^{19} c_ix_i \leq 0.
\end{equation} 
Now we set $c_1=1$, so that the coefficient of $\Omega$ is fixed. Next, for each $i$ such that (5.i) is an inequality rather than an equality, we add the constraint $c_i\geq 0$.
Finally, after expanding and collecting like terms in \eqref{Eq:LinearCombination}, we also constrain the coefficients of each term to be nonnegative.
For example, to ensure that the coefficient of the term $|S|$ is nonnegative in \eqref{Eq:LinearCombination}, we must add the constraint that $2 c_1- c_2+c_5\geq 0$.
We then maximize $c_2$ subject to these constraints in Mathematica, which gives the values in Table \ref{Table:CoefficientTable}.

\begin{table}[htbp]
    \centering
    \begin{tabular}{|l|l|l|l|l|}
        \hline
        $c_1=1$ &
        $c_{2}=99/37$ &
        $c_{3}=28/37$ &
        $c_{4}=28/37$ &
        $c_{5}=25/37$
        \\
        \hline
        $c_{6}=20/37$ &
        $c_{7}=25/37$ &
        $c_{8}=25/37$ &
        $c_{9}=4/37$ &
        $c_{10}=1/37$
        \\
        \hline
        $c_{11}=1/37$ &
        $c_{12}=1/37$ &
        $c_{13}=4/37$ &
        $c_{14}=1$ &
        $c_{15}=8/37$
        \\
        \hline
        $c_{16}=5/37$ &
        $c_{17}=8/37$ &
        $c_{18}=2/37$ & 
        $c_{19}=1/37$ &
        \\
         \hline
    \end{tabular}
    \caption{}
    \label{Table:CoefficientTable}
\end{table}

Expanding \eqref{Eq:LinearCombination} with these coefficients gives

\begin{align*}
    \frac{1}{37}\left(\abs{S_{3,1}^{S,S}}+\abs{S_{3,1}^{S\setminus f(S_1),T\cup \set{p_0}}}
    +\abs{S_{3,2}^{S\setminus f(S_1)}}+3\abs{S_{\geq 4,1}}+7\abs{S_{\geq 4,2}}+99\omega-187\right)-\Omega & \leq 0.
\end{align*}

Then, since each quantity is nonnegative, we can simplify to
\[
\OurLinearTerm\omega - \OurConstantTerm \leq \Omega,
\]
as desired.

To deal with the case when $3\nmid N$, we need to add the equality
\begin{equation}\label{Eq:No3s}
    f_3=0
\end{equation}
and change \eqref{Eq:DistinctFactors} to
\[
\omega = 1+\abs{S}+\abs{T}.
\]
Using the modified system of equations, we get the coefficients displayed in Table \ref{Table:No3CoefficientTable} using Mathematica.

\begin{table}[htbp]
    \centering
    \begin{tabular}{|l|l|l|l|l|}
        \hline
        $c_{1}=1$ &
        $c_{2}=51/19$ &
        $c_{3}=14/19$ &
        $c_{4}=10/19$ &
        $c_{5}=13/19$
        \\
        \hline
        $c_{6}=8/19$ &
        $c_{7}=13/19$ &
        $c_{8}=13/19$ &
        $c_{9}=4/19$ &
        $c_{10}=1/19$
        \\
        \hline
        $c_{11}=1/19$ &
        $c_{12}=1/19$ &
        $c_{13}=4/19$ &
        $c_{14}=21/19$ &
        $c_{15}=4/19$
        \\
        \hline
        $c_{16}=5/19$ &
        $c_{17}=0$ &
        $c_{18}=2/19$ &
        $c_{19}=1/19$ &
        $c_{21}=2/19$
        \\
         \hline
    \end{tabular}
    \caption{}
    \label{Table:No3CoefficientTable}
\end{table}

These coefficients lead to the linequality
\[
\frac{51}{19}\omega - \frac{46}{19} \leq \Omega,
\]
as claimed.

\section{Future work}

The improvements in this paper are based on the idea of linked primes. Further study into this concept could yield additional improvements.
In particular, showing that the linking map extends to a larger domain would increase the numerics of this paper, with a natural boundary of the linking map being $S_{4,2}$, due to the work in \cite{H&N}.

We were able to prove that a largest contributed prime can be shared among elements of $S_1\cup S_2 \cup S_{3,1}$ only twice, and ideally we would want to extend this restriction to larger subsets of $S$. 
However, this is not possible for $S_{3,2}$, as there exists a set of five primes that could be in $S_{3,2}$ that share their largest contributed prime (these are 120587, 269561, 324143, 473117, and 833033, which share 16963). 
We suspect that elements of $S_{3,2}$ could share their largest prime arbitrarily often.
Therefore, extending the linking map to $S_{3,2}$ would likely require a change in the definition of a linked prime.

Our work has mainly focused on improving the linear term in Theorem \ref{Theorem:OurResult}, so more work could be done to improve the constant term. In discussion with the authors, Ochem and Rao suggested a method to improve the constant term with casework. They were able to get an improved constant term of $-75/37$ under the assumption $e_0\geq 5$, and the term $-124/37$ when $e_0=1$ and $p_0\equiv 2\pmod 3$. However, in the final case, $e_0=1$ and $p_0\equiv 1\pmod 3$, it is not obvious how to proceed.

In proving his bound for when $3\nmid N$, Zelinsky uses the factorization
\[
\Phi_3(\Phi_5(x))=(x^2-x+1)(x^6+3x^5+5x^4+6x^3+7x^2+6x+3).
\]
It is not difficult to show that this factorization is part of a family of factorizations of compositions of cyclotomic polynomials or the product of cyclotomic polynomials.
Namely:
\begin{proposition}\label{Prop:Neat}
Let $\Psi_n(x)=\sum_{i=0}^{n-1}x^i$, let $r$ be an odd integer, and let $t$ be an odd prime. If $r\equiv -1\pmod{2t}$ then
\[
\Phi_{2t}(x)\,\mid \,\Phi_t(\Psi_r(x)).
\]
\end{proposition}

\begin{proof}
Observe that for any $2t^{\text{th}}$ root of unity $\zeta_{2t}$,
\[
\Psi_r(\zeta_{2t})=\Psi_{2t-1}(\zeta_{2t}).
\]
Then, note
\[
\Psi_{2t-1}(\zeta_{2t})=\zeta_{2t}^{t-1}.
\]
But $\zeta_{2t}^{t-1}=\zeta_t$, for some primitive $t^{\text{th}}$ root of unity. Thus,
\[
\Phi_t(\Psi_r(\zeta_{2t}))=\Phi_t(\zeta_t)=0.
\]
Since $\zeta_{2t}$ was arbitrary, all primitive $2t^{\text{th}}$ roots of unity are zeros of $\Phi_t(\Psi_r(x))$. Hence,
\[
\Phi_{2t}(x)\,\mid \,\Phi_t(\Psi_r(x)). \qedhere
\]
\end{proof}

Proposition \ref{Prop:Neat} implies that if $p^{e}\mid \mid N$ for some $e$ such that $e+1$ is prime and if $q=\sigma(p^e)$ is prime, then for even $m$ such that $q^m\mid \mid N$, then $\sigma(q^m)$ has at least two factors.
Thus, if there are many situations like this, we have even more total factors per distinct factor.

\section{Acknowledgments}
We thank our advisor Pace P. Nielsen for his invaluable assistance in preparing this paper. We thank Joshua Zelinsky, Pascal Ochem, and Micha\"{e}l Rao for comments that improved the paper.

\providecommand{\bysame}{\leavevmode\hbox to3em{\hrulefill}\thinspace}
\providecommand{\MR}{\relax\ifhmode\unskip\space\fi MR }
\providecommand{\MRhref}[2]{%
  \href{http://www.ams.org/mathscinet-getitem?mr=#1}{#2}
}
\providecommand{\href}[2]{#2}

\end{document}